\documentclass[12pt]{amsart} \usepackage{a4wide}

\usepackage{graphicx,xcolor}

\title[Waring and cactus ranks and SLP for annihilators of symmetric
forms]{Waring and cactus ranks and Strong Lefschetz Property for annihilators of symmetric forms}
\author[M.\
Boij]{Mats Boij} \address{Department of Mathematics, KTH Royal
  Institute of Technology, S-100 44 Stockholm, Sweden}
\email{boij@kth.se} \author[J.\ Migliore]{Juan Migliore}
\address{ Department of Mathematics, University of Notre Dame, Notre
  Dame, IN 46556, USA} \email{migliore.1@nd.edu} \author[R.\
M.\ Mir\'o-Roig]{Rosa M.\ Mir\'o-Roig} \address{
  Facultat de Matem\`atiques i Inform\`atica. Universitat de Barcelona. Gran
Via 585. 08007 Barcelona. Spain}\email{miro@ub.edu}
\author[U.\ Nagel]{Uwe Nagel} \address{Department of
  Mathematics, University of Kentucky, 715 Patterson Office Tower,
  Lexington, KY 40506-0027, USA} \email{uwe.nagel@uky.edu}

\newtheorem{thm}{Theorem}[section]
\newtheorem{prop}[thm]{Proposition}
\newtheorem{lemma}[thm]{Lemma}
\newtheorem{cor}[thm]{Corollary}

\theoremstyle{definition}
\newtheorem{dfn}[thm]{Definition}
\newtheorem{rmk}[thm]{Remark}
\newtheorem{notation}[thm]{Notation}

\newcommand{\ann}{\operatorname{ann}}
\thanks{The authors want to thank CIRM in Trento, MFO in Oberwolfach
  and the University of Kentucky for support and hospitality.
 The second author was partially supported by a Simons Foundation grant \#309556.
The third author was partially supported by MTM2016-78623-P. The fourth author was partially supported by Simons Foundation grants \#317096 and \#636513.}

\begin{document}
\maketitle
\begin{abstract}
  In this note we show that the complete symmetric
  polynomials are dual generators of compressed artinian Gorenstein
  algebras satisfying the Strong Lefschetz Property. This is the first
  example of an explicit dual form with these properties.

  For complete symmetric forms of any degree in any number of
  variables, we provide an upper bound for the Waring rank by
  establishing an explicit power sum decomposition. 

  Moreover, we determine the Waring rank, the cactus rank, the
  resolution and the
  Strong Lefschetz Property for any Gorenstein
  algebra defined by a symmetric cubic form. In particular, we show
  that the difference between the Waring rank and the cactus rank of a
  symmetric cubic form can be made arbitrarily large by increasing the
  number of variables.

  We provide upper bounds for the Waring rank of generic symmetric
  forms of degrees four and five. 
\end{abstract}


\section{Introduction}

Let $k$ be a field of characteristic $0$ and consider the polynomial
ring $S = k[x_1,x_2,\dots,x_n]$ and its inverse system
$E = k[X_1,X_2,\dots,X_n]$; more precisely, $E$ is an $S$-module where
$x_i$ acts as $\frac{\partial}{\partial X_i}$. It is a well-known fact
(\cite[Lemma 3.12]{IK}) that if $f$ is a general polynomial of degree $d$ in $E$ then the annihilator $J = \hbox{ann}(f)$ in $S$ defines an artinian Gorenstein algebra $S/J$ with compressed Hilbert function, meaning that the ``first half" of the Hilbert function of $S/J$ agrees with the Hilbert function of $S$, and the second half is of course determined by symmetry.

This begs the question of how ``general" $f$ has to be  in order to obtain a compressed Hilbert function in this way; in particular, can one exhibit a specific polynomial and prove that it has this property?
This problem arose in conversation during a Research in Pairs involving the authors, in Trento in 2009, while working on a different project. At that time, computer experiments showed that the complete symmetric polynomial (i.e. the sum of all the monomials of degree $d$) has this property. In the intervening years the authors were able to prove this fact, and in the process many related and intriguing problems presented themselves. This paper is the result.

In Section \ref{slp section}, in Theorem \ref{SLP result}  we give a proof of the fact just mentioned, that when $f = h_d$ is the complete symmetric polynomial then $S/\hbox{ann}(h_d)$ has compressed Hilbert function. (The fact that $S/\hbox{ann}(f)$ has compressed Hilbert function was also proven by
F.~Gesmundo and J.~M.~Landsberg \cite{GL}, with a different method.) In fact, Theorem \ref{SLP result} proves not only this, but that   $S/\hbox{ann}(h_d)$ even satisfies the {\it  Strong Lefschetz Property} (SLP). We recall that a graded artinian $k$-algebra $A$ satisfies the {\it Weak Lefschetz Property} (WLP) if multiplication by a general linear form $\times \ell$ from any component to the next has maximal rank, and it satisfies SLP if maximal rank also holds for $\times \ell^j : [A]_i \rightarrow [A]_{i+j}$ for all $i \geq 0$ and $j \geq 0$.  In the same section we also give a
useful decomposition  of $\hbox{ann}(h_d)$ (Theorem \ref{decomp}), separating into the cases where $d$ is even or odd. (The result for even degree is a bit stronger.)

  One of the main topics of this paper concerns the Waring rank, and we begin looking at this in Section
\ref{ps decomp section}. If $f$ is a homogeneous polynomial of degree
$m$, recall that its {\it Waring rank} is the smallest $k$ so that $f$
can be written in the form $L_1^m + \dots + L_k^m$ for linear forms
$L_1,\dots,L_k$, and a {\it power sum decomposition} is any expression
of $f$ in this form (not necessarily minimal). Let us denote by
$h_{n,m}$ the complete symmetric polynomial of degree $m$ in $n$
variables. The main result of Section \ref{ps decomp section}  is a
specific power sum decomposition for  $h_{n,2d+1}$ and separately for
$h_{n,2d}$, for any $n$ and $d$. In both the even and odd cases, our
power sum decomposition has $\binom{n+d}{d}$ terms, giving an upper
bound for the Waring rank. Thus our bound is a polynomial of degree
$d$ in $n$. Note that the Waring rank for general forms of degree
$\delta$ grows like a
polynomial of degree $\delta-1$ in $n$ (see Remark
\ref{rmk:bounds}). Later in the paper we will show that  our bound for
$h_{n,3}$ is sharp (see Theorem \ref{thm:powersums}), and we believe
that it is sharp for all odd degrees. In the case of even degrees,
we note that fewer terms may be possible at least
in some cases (see Remark \ref{rmk:bounds}).

In Section~\ref{cubic sect}, we consider only symmetric cubic
polynomials.
This is, in a sense, the heart of this paper. By choosing a very special basis for the vector space of symmetric forms of degree 3, we can view the family of all such polynomials as being parameterized by a projective plane, $\mathbb P^2$. Inside of this plane, we consider the forms that have Waring rank at most $n$. We show that this locus can be viewed as a rational cuspidal cubic curve, $\mathcal C$, and we give its equation (Lemma \ref{lemma:curve}).
From our perspective there are two important points on $\mathcal C$:
the (unique) flex point $\mathcal Q$   and the cusp $\mathcal P$. Then
there are three important lines: the line $\ell_1$ joining $\mathcal
P$ and $\mathcal Q$, the tangent line $\ell_2$ to $\mathcal C$ at
$\mathcal P$ (meaning the unique line that meets $\mathcal C$ at
$\mathcal P$ and no other point), and the tangent $\ell_3$ to $\mathcal C$ at
$\mathcal Q$.  We give the equations of these three lines.

With this preparation, we divide the points of $\mathbb P^2$ as
follows: $\mathcal P$, $\mathcal Q$, the remaining points of $\mathcal
C$, the remaining points of $\ell_1$, the remaining points of
$\ell_2$, the remaining points of $\ell_3$, and the remaining points of
$\mathbb P^2$. For each point of $\mathbb P^2$, depending on which of
these categories it falls into, we give the Hilbert function of the
corresponding algebra, a description of the generators of $\hbox{ann}
(F)$, the minimal free resolution, and both the Waring rank and the
cactus rank of the defining symmetric polynomial. Recall that the
{\em cactus rank} of a form $F$ is the smallest degree of a
zero-dimensional subscheme whose saturated ideal is contained in
$\hbox{ann}(F)$. We show as a consequence that for symmetric cubic
forms, the Waring rank and the cactus rank are equal except for points
on $\ell_2$ other than $\mathcal P$ (Corollary
\ref{compare}). We end this section by showing that for {\em any} symmetric cubic form $F$, the corresponding algebra $S/\hbox{ann}(F)$ satisfies the SLP (Proposition \ref{cubic SLP}).

In the last section, we look at the symmetric generic Waring rank,
that is the Waring rank of a generic symmetric form. We provide upper
bounds for this generic rank for quartics and quintics, generalizing
the results obtained for cubics in the previous section. These bounds
are close to the lower bounds given by the Hilbert function, but we
cannot establish the actual symmetric generic rank. 


\section{The SLP for the artinian Gorenstein algebra defined by a
  complete symmetric polynomial} \label{slp section}

Let $k$ be a field of characteristic $0$ and consider the polynomial
ring $S = k[x_1,x_2,\dots,x_n]$ and its inverse system
$E = k[X_1,X_2,\dots,X_n]$ which means that $E$ is an $S$-module where
$x_i$ acts as $\frac{\partial}{\partial X_i}$.

The \emph{homogeneous}, or \emph{complete}, symmetric polynomials are
defined as
\[
h_d(X_1,X_2,\dots,X_n) = \sum_{\underline i\in \mathbb N^n,
  |\underline i|=d} X_1^{i_1}X_2^{i_2}\cdots X_n^{i_n}.
\]

The goal of this section is to prove that the Gorenstein algebra $S/\operatorname{ann}(h_e)$ is compressed and satisfies the SLP (Theorem \ref{SLP result}).  The fact that it is compressed was also shown in \cite{GL}.

We fix $n$ and in the formal power series ring $k[[X_1,X_2,\dots,X_n]]$, which is
also a graded $S$-module, we define $H = \sum_{d=0}^\infty h_d =
\sum_{\underline i\in \mathbb N^n} X_1^{i_1}X_2^{i_2}\cdots
X_n^{i_n}$.
We will often use $\ell$ to denote the linear form $x_1+x_2+\dots+x_n$ in
$S$.

\begin{lemma}  The inverse
  system of the ideal $(\ell) = (x_1+x_2+\dots+x_n)$ is given by the
  subring $E'=k[X_1-X_n,X_2-X_n,\dots,X_{n-1}-X_n]$ considered as an
  $S$-module.
\end{lemma}

\begin{proof}
  The linear form $\ell$ annihilates any polynomial in
  $X_1-X_n,X_2-X_n,\dots,X_{n-1}-X_n$ and the dimension of the inverse
  system is $\binom{n-2+d}{n-2}$ in degree $d$ which equals the
  dimension of $E'_d$. Thus $E'$ is the inverse system of the ideal
  $(\ell)$.
\end{proof}

\begin{lemma}\label{lemma:ell_on_hd}
  $\ell\circ h_d = (n+d-1) h_{d-1}.$
\end{lemma}

\begin{proof}
  The monomial $X_1^{i_1}X_2^{i_2}\cdots X_n^{i_n}$ of degree $d-1$
  occurs from the derivatives
  $x_j\circ X_j X_1^{i_1}X_2^{i_2}\cdots X_n^{i_n} = (i_j+1)
  X_1^{i_1}X_2^{i_2}\cdots X_n^{i_n}$.
  The sum of these contributions is $n+\sum i_j =n+d-1$.
\end{proof}

\begin{dfn}
  Let $\Phi\colon S\longrightarrow E$ be the $k$-linear map defined on
  the monomial basis by
  \[\Phi(x_1^{i_1}x_2^{i_2}\cdots x_n^{i_n}) = i_1! i_2!\cdots i_n!
  X_1^{i_1}X_2^{i_2}\cdots X_n^{i_n}.\]
  for each $\underline i = (i_1,i_2,\dots,i_n)\in \mathbb N^n$.
\end{dfn}

\begin{rmk}
  Note that the inverse image of the monomial basis of $E_d$ gives the
  dual basis with respect to the action. Moreover, powers of $\ell$
  are sent to multiples of the homogeneous symmetric
  polynomials, $\Phi(\ell^d) = d!h_d$.
\end{rmk}

\begin{dfn}
  For any $\underline a = (a_1,a_2,\dots,a_{n-1})\in \mathbb N^{n-1}$
  we define
  \[F_{\underline a} =
  (X_1-X_n)^{a_1}(X_2-X_n)^{a_2}\cdots(X_{n-1}-X_n)^{a_{n-1}}
  \quad\text{and}\quad f_{\underline a} = \Phi^{-1}(F_{\underline
    a}).\]
\end{dfn}

\begin{lemma}\label{lemma_fa_o_H}
  For any $\underline a = (a_1,a_2,\dots,a_{n-1})$ we have that
  \[f_{\underline a} \circ H = \frac{F_{\underline
      a}}{(1-X_1)^{1+a_1}(1-X_2)^{1+a_2} \cdots
    (1-X_{n-1})^{1+a_{n-1}}(1-X_n)^{1+\sum_{i=1}^{n-1} a_i}}.\]
\end{lemma}

\begin{proof}
  We expand $f_{\underline a}$ by the binomial theorem as
  \begin{multline*}
    \Phi^{-1}\left( \sum_{\underline j\in \mathbb N^{n-1},\underline
        j\le \underline a} (-1)^{\sum_{i=1}^{n-1} a_i-j_i}
      \binom{a_1}{j_1}\binom{a_2}{j_2}\cdots \binom{a_{n-1}}{j_{n-1}}
      X_1^{j_1}X_2^{j_2}\cdots X_{n-1}^{j_{n-1}} X_n^{\sum_{i=1}^{n-1} a_i-j_i}\right) \\
    = \sum_{\underline j\in \mathbb N^{n-1},\underline j\le \underline
      a} (-1)^{\sum_{i=1}^{n-1} a_i-j_i}
    \binom{a_1}{j_1}\binom{a_2}{j_2}\cdots \binom{a_{n-1}}{j_{n-1}}
    \frac{x_1^{j_1}x_2^{j_2}\cdots x_{n-1}^{j_{n-1}} x_n^{\sum_{i=1}^{n-1} a_i-j_i} }{j_1!j_2!\cdots j_{n-1}!(\sum_{i=1}^{n-1} a_i-j_i)!}.\\
  \end{multline*}
  Now look at the contribution to $f_{\underline a}\circ H$ by each
  term of this sum. The only monomials in $H$ that contribute are the
  multiples of the corresponding monomial in the variables
  $X_1,X_2,\dots,X_n$.  We get that
  \begin{multline*}
    x_1^{j_1}x_2^{j_2}\cdots x_{n-1}^{j_{n-1}}x_n^{\sum_{i=1}^{n-1}
      a_i-j_i} \circ H = x_1^{j_1}x_2^{j_2}\cdots
    x_{n-1}^{j_{n-1}}x_n^{\sum_{i=1}^{n-1} a_i-j_i} \circ
    X_1^{j_1}X_2^{j_2}\cdots X_{n-1}^{j_{n-1}}X_n^{\sum_{i=1}^{n-1}
      a_i-j_i} H \\= \sum_{\underline b\in \mathbb N^n}
    x_1^{j_1}x_2^{j_2}\cdots x_{n-1}^{j_{n-1}}x_n^{\sum_{i=1}^{n-1}
      a_i-j_i} \circ X_1^{b_1+j_1}X_2^{b_2+j_2}\cdots
    X_{n-1}^{b_{n-1}+j_{n-1}}X_n^{b_n+\sum_{i=1}^{n-1} a_i-j_i} \\=
    \sum_{\underline b\in \mathbb N^n} \frac{(b_1+j_1)!}{b_1!}
    \frac{(b_2+j_2)!}{b_2!}  \cdots
    \frac{(b_{n-1}+j_{n-1})!}{b_{n-1}!}  \frac{(b_n+\sum
      a_i-j_i)!}{b_n!}  X_1^{b_1}X_2^{b_2}\cdots X_n^{b_n} \\=
    \frac{j_1!}{(1-X_1)^{1+j_1}} \frac{j_2!}{(1-X_2)^{1+j_2}} \cdots
    \frac{j_{n-1}!}{(1-X_{n-1})^{1+j_{n-1}}} \frac{(\sum
      a_i-j_i)!)}{(1-X_n)^{1+\sum_{i=1}^{n-1} a_i-j_i}}.
  \end{multline*}
  Summing over the terms of $f_{\underline a}$ now gives
  \begin{multline*}
    f_{\underline a} \circ H = \sum_{\underline j \in \mathbb N^{n-1},
      \underline j\le \underline a}
    \frac{(-1)^{\sum_{i=1}^{n-1} a_i-j_i}\binom{a_1}{j_1}\binom{a_2}{j_2}\cdots
      \binom{a_{n-1}}{j_{n-1}}}{(1-X_1)^{1+j_1}(1-X_2)^{1+j_2}\cdots(1-X_{n-1})^{1+j_{n-1}}(1-X_n)^{1+\sum
        a_i-j_i}} \\= \left(1 - \frac{1-X_n}{1-X_1}\right)^{a_1}
    \left(1 - \frac{1-X_n}{1-X_2}\right)^{a_2} \cdots \left(1 -
      \frac{1-X_n}{1-X_{n-1}}\right)^{a_{n-1}} \frac{1}{(1-X_n)^{\sum
        a_i}} \frac{(-1)^{\sum_{i=1}^{n-1} a_i}}{\prod (1-X_i)}
    \\=(-1)^{\sum_{i=1}^{n-1} a_i}
    \frac{(X_n-X_1)^{a_1}}{(1-X_1)^{1+a_1}}
    \frac{(X_n-X_2)^{a_2}}{(1-X_2)^{1+a_2}} \cdots
    \frac{(X_n-X_{n-1})^{a_{n-1}}}{(1-X_{n-1})^{1+a_{n-1}}}
    \frac{1}{(1-X_n)^{1+\sum a_i}}\\= \frac{F_{\underline a}
    }{(1-X_1)^{1+a_1}(1-X_2)^{1+a_2} \cdots
      (1-X_{n-1})^{1+a_{n-1}}(1-X_n)^{1+\sum_{i=1}^{n-1} a_i}}
  \end{multline*}
  which concludes the proof of the lemma.
\end{proof}

\begin{lemma}\label{lemma_fa_on_h2d}
  For $\underline a = (a_1,a_2,\dots,a_{n-1})$ in $\mathbb N^{n-1}$
  with $|\underline a|=d$ we have that
  \begin{enumerate}
  \item $f_{\underline a} \circ h_i = 0$, for $i<2d$.
  \item $f_{\underline a} \circ h_{2d} = F_{\underline a}$.
  \item
    $f_{\underline a} \circ h_{2d+m} = F_{\underline a}
    G_{m,\underline a}$, where
    \[G_{m,\underline a} = \sum_{\underline i\in \mathbb N^n,
      |\underline i|=m} \binom{a_1+i_1}{i_1} \binom{a_2+i_2}{i_2}
    \cdots \binom{a_{n-1}+i_{n-1}}{i_{n-1}} \binom{\sum a_j +
      i_n}{i_n} X_1^{i_1}X_2^{i_2}\cdots X_n^{i_n}.\]
  \end{enumerate}
\end{lemma}

\begin{proof}
  The first two statements follow from the fact that the initial term in
  $f_{\underline a}\circ h_{2d}$ is $F_{\underline a}$ which has degree $d$. The
  last statement is given by the power series expansion of
  \[
  \frac{1} {(1-X_1)^{1+a_1}(1-X_2)^{1+a_2} \cdots
    (1-X_{n-1})^{1+a_{n-1}}(1-X_n)^{1+\sum_{i=1}^{n-1} a_i}}.
  \]
\end{proof}

\begin{lemma} For $\underline a = (a_1,a_2,\dots,a_{n-1})$ and
  $\underline b = (b_1,b_2,\dots,b_{n-1})$ in $\mathbb N^{n-1}$ with
  $|\underline a|=|\underline b|=d$, we have that
  \[
  f_{\underline a} \circ F_{\underline b} = \prod_{i=1}^{n-1}
  \binom{a_i+b_i}{a_i}.
  \]
\end{lemma}

\begin{proof}
  We have
  \begin{multline*}
    f_{\underline a}\circ F_{\underline b} = \sum_{\underline i\in
      \mathbf N^{n-1}}
    \prod_{k=1}^{n-1}\frac{x_j^{i_k}}{i_k!}\binom{a_k}{i_k}
    \frac{(-x_n)^{d-|\underline i|} }{(d-|\underline i|)!}  \circ
    \sum_{\underline j\in \mathbf N^{n-1}}
    \prod_{k=1}^{n-1}x_k^{j_k}\binom{b_k}{j_k} (-x_n)^{d-|\underline
      j|}
    \\
    = \sum_{\underline i\in \mathbf N^{n-1}}
    \prod_{k=1}^{n-1}\binom{a_k}{i_k} \binom{b_k}{i_k} =
    \prod_{k=1}^{n-1} \sum_{i=0}^\infty \binom{a_k}{i} \binom{b_k}{i}
    = \prod_{i=1}^{n-1} \binom{a_i+b_i}{a_i}
  \end{multline*}
  where we use that
  \[
  \sum_{i=0}^{\infty} \binom{a}{i}\binom{b}{i} = \binom{a+b}{a}
  \]
  which comes from looking at the coefficient of $t^a$ in
  $(1+t)^{a+b} = (1+t)^a(1+t)^b$.
\end{proof}
\begin{dfn}
  For integers $d\ge 0$ let
  $M_d = \Phi^{-1} (E'_d) = \langle\{ f_{\underline a}\colon
  \underline a\in \mathbb N^{n-1}, |\underline a|=d\}\rangle .$
\end{dfn}

\begin{prop}\label{prop:pairing}
  The pairing $\ell^i M_{d-i}\times \ell^j M_{d-j} \longrightarrow k$
  given by $\langle f,g\rangle = (fg)\circ h_{2d}$ is trivial when
  $i\ne j$ and is perfect when $i=j$.
\end{prop}

\begin{proof}
  We have that a basis for $\ell^i M_{d-i}$ is given by
  $\ell^i f_{\underline a}$ with $|\underline a|=d-i$. We have that
  \[
  (\ell^{i}f_{\underline a} \cdot \ell^jf_{\underline b} )\circ h_{2d}
  = \frac{(n+2d-1)!}{(n+2d-i-j-1)!} f_{\underline a}f_{\underline b}
  \circ h_{2d-i-j}.
  \]
  If $i\ne j$, we have that either $f_{\underline a}$ or
  $f_{\underline b}$ annihilates $h_{2d-i-j}$ by
  Lemma~\ref{lemma_fa_on_h2d}(1) since either $2(d-i)>2d-i-j$ or
  $2(d-j)>2d-i-j$. Hence the pairing is trivial.

  If $i=j$, it suffices to show that the pairing
  $M_{d-i}\times M_{d-i}\longrightarrow k$ given by
  $\langle f,g\rangle = (fg)\circ h_{2(d-i)}$ is perfect. By
  Lemma~\ref{lemma_fa_on_h2d} (2) we get that for $f$ in $M_{d-i}$ we
  have $\Phi(f) = f\circ h_{2(d-i)}$. Hence the pairing is perfect
  since $\Phi\colon S_{d-i}\longrightarrow E_{d-i}$ is invertible.
\end{proof}

\begin{thm} \label{SLP result}
  For each $e\ge 0$ the Gorenstein algebra
  $A = S/\operatorname{ann}(h_e)$ is compressed and satisfies the SLP.
\end{thm}

\begin{proof}
  For even socle degree, $e = 2d$, we need to show that the ideal
  $\operatorname{ann}(h_{2d})$ is zero in degree $d$ in order to
  conclude that the algebra is compressed. In degree $d$, we have that
  $ M_d+ \ell M_{d-1} + \cdots + \ell^dM_0\subseteq S_d$. Using
  Proposition~\ref{prop:pairing} we see that the pairing given by
  $\langle f,g\rangle = (fg)\circ h_{2d}$ is perfect on each of the
  summands and trivial between any two of them. Since the sum of the
  dimensions of the summands equals the dimension of $S_d$, we
  conclude that the pairing is perfect on $S_d$, which implies that
  $\operatorname{ann}(h_{2d})$ is trivial in degree $d$.

  For odd socle
  degree, $e=2d+1$, it is enough to show that
  $\operatorname{ann}(h_{2d+1})$ is zero in degree $d$. If $f$ has
  degree $d$ and satisfies $f\circ h_{2d+1}=0$ we also have
  $(\ell f)\circ h_{2d+2}$ by Lemma~\ref{lemma:ell_on_hd}, but we have
  shown that $\operatorname{ann}(h_{2d+2})$ is trivial in degree $d+1$
  and therefore $\ell f = f = 0$. Hence the algebra is compressed.

  Let $A = S/\operatorname{ann}(h_{e})$. In order to prove that $A$
  has the strong Lefschetz property, it is sufficient to show that
  $\ell^{e-2i}\colon A_i\longrightarrow A_{e-i}$ is an isomorphism for
  all $i\le e/2$. Assume that $f\in S_i$ satisfies $\ell^{e-2i}f =0$
  in $A_{e-i}$. This means that $\ell^{e-2i}f \circ h_{e}=0$, which
  by Lemma~\ref{lemma:ell_on_hd} means that $f\circ h_{2i}=0$.
  However, since $A/\operatorname{ann}(h_{2i})$ is compressed, this
  means that $f=0$ and the multiplication by $\ell^{e-2i}$ is
  injective and by symmetry of the Hilbert function also bijective.
\end{proof}

\begin{thm} \label{decomp}
  For even socle degree, $e = 2d$, we have that
  $\operatorname{ann}(h_{2d}) = (M_{d+1}\oplus \ell M_d)$ and for odd
  socle degree, $e = 2d+1$, we have that
  $\operatorname{ann}(h_{2d+1}) = (M_{d+1}+\ell^2 M_d)$.
\end{thm}

\begin{proof}
  Since the algebra $A =S/\operatorname{ann}(h_{2d})$ is compressed
  of even socle degree we have generators of
  $\operatorname{ann}(h_{2d})$ only in degree $d+1$ and the number of
  generators is given by
  \[\binom{n+d}{n-1}-\binom{n+d-2}{n-1} =
  \binom{n+d-1}{n-2}+\binom{n+d-2}{n-2}.\]
  Both $M_{d+1}$ and $\ell M_d$ annihilate $h_{2d}$ and their
  intersection is trivial as we can see using the pairing given by
  $\langle f,g\rangle = (fg)\circ h_{2d+2}$.  Indeed, from
  Proposition~\ref{prop:pairing} we have that for a non-zero element
  $f$ in $M_{d+1}$, there is an element $g\in M_{d+1}$ with $\langle
    f,g\rangle \ne 0$, while if $f\in \ell M_d$, we have to have
    $\langle f,g\rangle =0$ from the same proposition.

  Moreover,
  $\dim_k M_{d+1} = \binom{n+d-1}{n-2}$ and
  $\dim_k \ell M_{d} = \binom{n+d-2}{n-2}$, which proves that
  $\operatorname{ann}(h_{2d}) = (M_{d+1}\oplus \ell M_d)$.

  Since the algebra $A =S/\operatorname{ann}(h_{2d+1})$ is compressed
  of odd socle degree and we can have generators in degree $d+1$ and
  $d+2$. The dimension of $\operatorname{ann}(h_{2d+1})$ in degree
  $d+1$ is given by
  $\binom{n+d}{n-1}-\binom{n+d-1}{n-1} = \binom{n+d-1}{n-2}$ which
  equals $\dim_k M_{d+1}$. Since $M_{d+1}$ annihilates $h_{2d+1}$ we
  must have that $\operatorname{ann}(h_{2d+1})$ in degree $d+1$ equals
  $M_{d+1}$. Let $g\in S_{d+2}$ be an element of
  $\operatorname{ann}(h_{2d+1})$. Since we have that
  $M_{d+1}\oplus (\ell)_{d+1} = S_{d+1}$ we must have that
  $g = \ell f_1 + f_2 $ for some $f_1\in S_{d+1}$ and $f_2$ in
  $(M_{d+1})_{d+2}$. Now $g\circ h_{2d+1} = 0$ implies that
  $f_1\circ h_{2d} =0$. From above it follows that
  $\operatorname{ann}(h_{2d})_{d+1} = M_{d+1}\oplus \ell M_d$ so
  $f_1 \in \ell M_d+ M_{d+1}$. Thus we see that
  $g\in \ell^2 M_d + (M_{d+1})$, which concludes the proof.
\end{proof}


\section{A power sum decomposition for the complete symmetric
  polynomial} \label{ps decomp section}

In this section we will give a decomposition of the complete symmetric
polynomial as a sum of powers of linear forms. The number of terms in
this decomposition is in general much lower than the Waring rank of a
general polynomial of the same degree (see Remark~\ref{rmk:bounds}) and it seems likely that our
decomposition is a Waring decomposition when the number of
variables is higher than the degree and the degree is odd. We will prove this in the case of
degree three.

\begin{lemma}\label{lem:diffrestrict}
  For a polynomial $f\in \mathbb Q[X_1,X_2,\dots,X_n]$ we have that
  \[
  f = 0 \qquad \Longleftrightarrow \qquad
  \begin{cases}
    f|_{X_n=0} = 0, \text{ and} \\
    (x_1+x_2+\dots+x_n)\circ f = 0.
  \end{cases}
  \]
\end{lemma}

\begin{proof}
   Assume that
  $f|_{X_n=0}=0$ and $\ell \circ f=0$, where
  $\ell = x_1+x_2+\cdots+x_n$. Then $f = 0$ or $f = X_n^d g$, where
  $g|_{X_n=0}\ne 0$ and $d>0$. Applying $\ell\circ f =0$ we get
  \[
  dX_n^{d-1}g + X_n^d \ell\circ g = 0
  \]
  and division by $X_n^{d-1}$ shows that $g|_{X_n=0} = 0$. Hence
  $f=0$. The other direction is trivial.
\end{proof}

In the following we will use the notation $h_{n,d}$ for the complete symmetric
polynomial of degree $d$ in the variables $X_1,X_2,\dots,X_n$.

\begin{thm}\label{thm:powersums} For $n\ge 1$ and $d\ge 0$ the following power sum
  decompositions hold for the complete symmetric functions
  \[
  h_{n,2d+1}=\frac{1}{2^{2d}(2d+1)!}\sum_{k=0}^d
  (-1)^{d-k}\binom{n+2d}{d-k} \sum_{1\le i_1\le i_2\le\cdots\le
    i_{k}\le n} \left(h_{n,1} + 2 \sum_{j=1}^{k} X_{i_j}\right)^{2d+1}
  \]
  and
  \[h_{n,2d}=\frac{1}{2^{2d}(2d)!}\sum_{k=0}^d(-1)^{d-k}\left[\binom{n+2d-1}{d-k}-\binom{n+2d-1}{d-k-1}\right]\sum_{1\le
    i_1\le i_2\le\cdots\le i_{k}\le n} \left(h_{n,1} + 2 \sum_{j=1}^{k}
    X_{i_j}\right)^{2d}
\]
where in both cases the second summation contains just one term
$h_{n,1}^{2d+1}$ and $h_{n,1}^{2d}$, respectively, for $k=0$.
\end{thm}

\begin{proof}
  We will prove this by induction on $n$ and $d$ and we introduce the
  notation
  \[
  S_{n,k,d}=\sum_{1\le i_1\le i_2\le\cdots\le i_{k}\le n}
  \left(h_{n,1}+2\sum_{j=1}^{k}X_{i_j}\right)^{d}
  \]
  as an expression valid in any polynomial ring
  $\mathbb Q[X_1,X_2,\dots,X_N]$ with $N\ge n$.  Furthermore, for
  $m\ge 0$, we introduce the notation
  \[
  \begin{split}
    A_{n,d,m}&=\sum_{k=0}^d (-1)^{d-k}\binom{n+2d}{d-k} S_{n,k,2m+1}\\
    B_{n,d,m}&=\sum_{k=0}^d(-1)^{d-k}\left[\binom{n+2d-1}{d-k}-\binom{n+2d-1}{d-k-1}\right]S_{n,k,2m}
  \end{split}
  \]
  for $n\ge 1$, $d\ge 0$ and $m\ge 0$. We shall prove by induction
  also on $m\le d$ that
  \begin{equation}
    A_{n,d,m}=\begin{cases}2^{2d}(2d+1)! h_{n,2d+1},& m=d,\\0,&m<d\end{cases}
    \quad\text{and}\quad
    B_{n,d,m}=\begin{cases}2^{2d}(2d)!h_{n,2d},&
      m=d,\\0,&m<d.\end{cases}\label{eq:1}
  \end{equation}
  The statement of the theorem is (\ref{eq:1}) in the case $m=d$.

  The base of the induction is given by the cases $d=0$, $n=1$ and
  $m=0$.  For $d=0$ we have that $m=0$ and we get
  \[
  A_{n,0,0} = S_{n,0,1} = \sum_{i=1}^n X_i = 2^01!
  h_{n,1}\quad\text{and}\quad B_{n,0,0} = S_{n,0,0} = 1 = 2^00!
  h_{n,0}
  \]
  so the assertions hold for $d=0$. For $n=1$ we have that
  $S_{1,k,d} = (1+2k)^dX_1^d$ and we get
  \[
  \begin{split}
    A_{1,d,m} &= \sum_{k=0}^d (-1)^{d-k} \binom{1+2d}{d-k}
    (1+2k)^{2m+1}X_1^{2m+1} = \sum_{k=0}^d (-1)^{k} \binom{1+2d}{k}
    (1+2(d-k))^{2m+1}X_1^{2m+1} \\&= \frac12 \sum_{k=0}^{2d+1}
    (-1)^{k} \binom{1+2d}{k} (1+2d-2k)^{2m+1} X_1^{2m+1}=
    \begin{cases}
      2^{2d}(2d+1)!X_1^{2d+1}, & m=d,\\0,&0\le m<d
    \end{cases}
  \end{split}
  \]
  and
  \[
  B_{1,d,m} = \frac{1}{(2d+1)(2m+1)} x_1\circ A_{1,d,m} =
  \begin{cases}
    2^{2d}(2d)!X_1^{2d}, & m=d,\\0,&0\le m<d
  \end{cases}
  \]
  according to the previous equality. Thus both equalities hold for
  $n=1$.

  For $m=0$ we have that
  \[
  \begin{split}
    A_{n,d,0} &= \sum_{k=0}^d (-1)^{d-k} \binom{n+2d}{d-k} S_{n,k,1}
    =h_{n,1}\sum_{k=0}^d (-1)^{d-k}
    \binom{n+2d}{d-k}\binom{n-1+k}{k}\frac{n+2k}{n} \\&
    =h_{n,1}\sum_{k=0}^d (-1)^{d-k} \binom{n+2d}{d-k}\left[\binom{n+k}{k}
      + \binom{n+k-1}{k-1}\right]
  \end{split}
  \]
  and the coefficient of $h_{n,1}$ equals the coefficient of $t^d$ in the
  formal power series
  \[
  (1-t)^{n+2d}\left[\frac{1}{(1-t)^{n+1}}+\frac{t}{(1-t)^{n+1}}\right]
  = (1-t)^{2d-1}+t(1-t)^{2d-1}
  \]
  which is equal to
  \[
  (-1)^d \left[\binom{2d-1}{d} - \binom{2d-1}{d-1}\right] =
  \begin{cases} 1,&d=0,\\0,&d>0.\end{cases}
  \]

  For $B_{n,d,0}$ we get
  \[
  B_{n,d,0} = \sum_{k=0}^d (-1)^{d-k} \left[\binom{n+2d-1}{d-k} -
    \binom{n+2d-1}{d-1-k}\right]\binom{n-1+k}{k}
  \]
  which is the coefficient of $t^d$ in the formal power series
  \[
  \left[(1+t)^{n+2d-1} -t(1+t)^{n+2d-1}\right]\cdot \frac1{(1+t)^n} =
  (1+t)^{2d-1}-t(1+t)^{2d-1}
  \]
  which equals
  \[
  \binom{2d-1}{d}-\binom{2d-1}{d-1} = \begin{cases} 1,&d=0,\\0,&d>0.
  \end{cases}
  \]

  Now we proceed to the induction step. By
  Lemma~\ref{lem:diffrestrict} it is sufficient to show that the
  equalities in (\ref{eq:1}) hold after differentiation by $\ell=x_1+x_2+\dots+x_n$
  and after restriction to $X_n=0$.

  We have that
  \[
  \begin{split}
    \ell\circ A_{n,d,m} &= (2m+1)\sum_{k=0}^d
    (-1)^{d-k}\binom{n+2d}{d-k} (n+2k)S_{n,k,2m} = (2m+1)(n+2d)
    B_{n,d,m} \\&=
    \begin{cases}
      2^{2d}(2d+1)!(n+2d)h_{n,2d},& m =d\\
      0,&m<d
    \end{cases}
  \end{split}
  \]
  where we have to assume by induction that the equality holds for
  $B_{n,d,m}$.  We have that
  \[
  \ell\circ B_{n,d,m} = \frac{1}{(2m+1)(n+2d)} \ell^2\circ A_{n,d,m} =
  \frac{2m}{n+2d}\sum_{k=0}^d
  (-1)^{d-k}\binom{n+2d}{d-k}(n+2k)^2S_{n,k,2m-1}.
  \]
  We subtract $2m(n+2d) A_{n,d,m-1}$, which by induction on $m$ is
  zero, from this and get
  \[
  \begin{split}
    \ell\circ B_{n,d,m} &= \frac{2m}{n+2d}\sum_{k=0}^{d-1}
    (-1)^{d-k}\binom{n+2d}{d-k}\left[(n+2k)^2-(n+2d)^2\right]S_{n,k,2m-1}
    \\&= \frac{2m}{n+2d}\sum_{k=0}^{d-1}
    (-1)^{d-k}\binom{n+2d}{d-k}(2k-2d)(2n+2k+2d)S_{n,k,2m-1} \\&=
    8m(n+2d-1)\sum_{k=0}^{d-1}
    (-1)^{d-1-k}\binom{n+2d-2}{d-1-k}S_{n,k,2m-1} \\&= 8m(n+2d-1)
    A_{n,d-1,m-1} =
    \begin{cases}
      8d(n+2d-1)2^{2d-2}(2d-1)! h_{n,2d-1}&m=d\\
      0&m<d
    \end{cases}
    \\&=
    \begin{cases}
      (n+2d-1)2^{2d}(2d)! h_{n,2d-1}&m=d\\
      0&m<d
    \end{cases}
  \end{split}
  \]
  which agrees with the derivative of the right-hand side of (\ref{eq:1}).

  Next, we restrict to $X_n=0$ where we have
  $S_{n,k,d}|_{X_n=0} = \sum_{i=0}^k S_{n-1,i,d}$. Hence
  \[
  \begin{split}
    A_{n,d,m}|_{X_n=0} &= \sum_{k=0}^d (-1)^{d-k}\sum_{i=0}^k
    \binom{n+2d}{d-k}S_{n-1,i,2m+1} = \\&= \sum_{i=0}^d
    \sum_{j=0}^{d-i} (-1)^{j} \binom{n+2d}{j}S_{n-1,i,2m+1} =
    \sum_{i=0}^d (-1)^{d-i} \binom{n-1+2d}{d-i}S_{n-1,i,2m+1} \\&=
    A_{n-1,d,m}
  \end{split}
  \]
  and
  \[
  \begin{split}
    B_{n,d,m}|_{X_n=0} &= \sum_{k=0}^d (-1)^{d-k}\sum_{i=0}^k
    \left[\binom{n+2d-1}{d-k}-\binom{n+2d-1}{d-1-k}\right]S_{n-1,i,2m}
    \\&= \sum_{i=0}^d \sum_{j=0}^{d-i} (-1)^{j}
    \left[\binom{n+2d-1}{j}-\binom{n+2d-1}{j-1}\right]S_{n-1,i,2m}
    \\&= \sum_{i=0}^d \sum_{j=0}^{d-i} (-1)^{d-i}
    \left[\binom{n+2d-1}{d-i}-\binom{n+2d-1}{d-i-1}\right]S_{n-1,i,2m}
    = B_{n-1,d,m}.
  \end{split}
  \]
\end{proof}

\begin{rmk}\label{rmk:bounds} For either even degree $2d$ or odd degree $2d+1$, our power sum
  decomposition in Theorem~\ref{thm:powersums} has $\binom{n+d}{d}$ terms. This gives an upper bound for the Waring rank of the complete
  symmetric forms. 
  According to the
  Alexander--Hirschowitz Theorem, except in a few well-understood
  cases~\cite{AH,I}, general forms of degree $2d+1$ have Waring rank 
  $\lceil\frac1n\binom{n+2d}{2d+1}\rceil$ and general forms of degree
  $2d$ have Waring rank $\lceil\frac1n\binom{n+2d-1}{2d}\rceil$.
  Observe
  that our bound
  has degree $d$ as a polynomial in $n$ while the Waring rank for
  general forms grows like  a polynomial of degree $2d$ or $2d-1$ in
  $n$. We believe that our bound is sharp for odd degree (i.e. that we
  are actually giving a Waring decomposition), and will show it for
  $d=1$ (i.e. degree 3) in the next section. For even degree,
  experimentally we have seen that it is not necessarily sharp.

  In fact, for all $n\ne 14$, we can find coefficients
  $\alpha_0,\alpha_1,\alpha_2,\alpha_3,\alpha_4$ so that
  \[
  h_{n,4} = \alpha_0 h_1^4 + \sum_{i=1}^{n}
  (\alpha_1 h_1+\alpha_2 X_i)^4
  + \sum_{1\le i < j \le n} (\alpha_3 h_1 + \alpha_4 X_i+\alpha_4
  X_j)^4.
  \]
  This expression has $\binom{n+2}{2}-n$ terms which is $n$ fewer
  terms than the expression from Theorem.~\ref{thm:powersums}.
  In particular, for $n=13$ we can get explicit values for the
  coefficients and we can write
  \[
    h_{13,4} = -\frac{2}{3} h_1^4 + \frac{1}{24}\sum_{i=1}^{13}
    (h_1+X_i)^4  + \frac{1}{24}\sum_{1\le i < j \le 13} (X_i+X_j)^4.
  \]
  We will look more into this at the end of Section~\ref{section:genWR}.
\end{rmk}


\section{Symmetric cubic polynomials}
  \label{cubic sect}

In this section we deeply analyze the case of symmetric cubic
polynomials. We begin by considering the Waring rank. In the case of
degree three, Theorem~\ref{thm:powersums} applied with $d=1$ states that
\[
\displaystyle h_3 = \frac1{24} \left ( \sum_{i=1}^n (h_1 + 2X_i)^3
  -(n+2) h_1^3\right)
\]
 and we will show that this is indeed a Waring decomposition, not just
for this particular symmetric cubic, but that a general symmetric cubic has
a similar Waring decomposition.

We will study the artinian Gorenstein algebras defined by the
annihilators of symmetric cubic forms in $n \ge 3$ variables. We will
determine the possible Hilbert function of the annihilator, the Waring
rank, the cactus rank, and the resolution of all such Gorenstein
algebras. We will also show that they all satisfy the SLP and that
there are three linear forms that are sufficient to provide Lefschetz
elements for them all.

We will parameterize the symmetric forms of degree three by the projective
plane using the basis $\{p_1^3,np_1p_2,n^2p_3\}$, where $p_i =
\sum_{j=1}^n X_j^i$ are the power sum symmetric forms in $E =
k[X_1,X_2,\dots,X_n]$. Let $k[a_0,a_1,a_2]$ be the coordinate ring of
this plane corresponding to the given basis.

We start by defining a rational cubic curve as a set of symmetric cubics whose Waring rank
is at most $n$.  For $(\alpha\colon\beta)$ in $\mathbb P^1$,  define (up to a scalar multiple)
a symmetric cubic
\begin{align}
  \label{eq:parametrization}
F(\alpha,\beta) & = \sum_{i=1}^n (\alpha n X_i+\beta p_1)^3  \\
& =\alpha^3n^3p_3+
3\alpha^2\beta n^2p_1p_2 + (3\alpha\beta^2+\beta^3)np_1^3. \nonumber
\end{align}
Furthermore, let $\mathcal C$ be the rational cubic curve defined as the image of
the map
\[
\gamma\colon \mathbb P^1 \to \mathbb P^2 =
\mathbb P(\langle p_1^3, np_1p_2,n^2p_3\rangle), \ (\alpha\colon\beta)
\mapsto (3\alpha\beta^2+\beta^3 : 3\alpha^2\beta : \alpha^3).
\]
(Note the coefficients  $n$ and $n^2$ on the basis.) That means, the image of $(\alpha\colon\beta)$ corresponds to the cubic $F (\alpha, \beta)$.

We also will use the above coordinates to parameterize any symmetric cubic  in $n$ variables.
In fact, since the power sums whose degrees are at most $n$ generate the ring of symmetric polynomials, every symmetric cubic can be written (over a field of characteristic zero) as $F = a_0p_1^3+a_1np_1p_2+a_2n^2p_3$. We refer to $F$ as the cubic at the point $(a_0 \colon a_1 \colon a_2)$.

\begin{notation} We fix three lines in $\mathbb P^2$ which will play an important role in this section:

$\ell _1$ is the line defined by $a_1+3a_2=0$,

$\ell _2$ is the line defined by $a_2=0$, and

 $\ell _3$ is the line defined by $a_0 + a_1 + a_2 = 0$.
\end{notation}

\begin{lemma}\label{lemma:curve}
  The cubic curve $\mathcal C$ is a cuspidal curve given by the
  equation
\[
    27 a_0a_2^2  - 9a_1^2a_2-a_1^3 = 0.
  \]
  Its cusp is $\mathcal P = (1:0:0)$ and its unique flex point is
  $\mathcal Q = (2:-3:1)$. The line  through $\mathcal P$ and  $\mathcal Q$ is the line $\ell _1$.
  Moreover, $\ell _2$ is the unique line meeting  $\mathcal C$ at $\mathcal P$ with multiplicity three.
\end{lemma}

\begin{proof} It is elementary to check that $\mathcal C$ has a singularity at $(1:0:0)$. This is a cusp rather than a node because there is only one point $(\alpha : \beta)$ in $\mathbb P^1$ such that $\gamma((\alpha:\beta)) = (1:0:0)$, namely $(0:1)$.
The fact that a cuspidal cubic has only one flex point can be found in \cite[page 146]{BCGM}. One can check that the tangent line  at $\mathcal Q = (2:-3:1)$ is the line $\ell _3$  and that the intersection of this line with $\mathcal C$ at $\mathcal Q$ has multiplicity  3. Hence $\mathcal Q$ is indeed a flex point. The rest is immediate.
\end{proof}

Below we slightly abuse notation and refer to the line $\ell _2$ as the \emph{tangent line to $\mathcal C$ at $\mathcal P$}.

\begin{prop}\label{prop:HF}
  The Hilbert function of $S/\ann(a_0p_1^3+a_1np_1p_1+a_2n^2p_3)$ is
  \[
    \begin{cases}
      (1,1,1,1)& \text{at $\mathcal P = (1:0:0)$}\\
      (1,n-1,n-1,1)& \text{at $\mathcal Q = (2,-3,1)$}\\
      (1,n,n,1)& \text{otherwise.}
    \end{cases}
  \]
\end{prop}

\begin{proof}
  Since $\ann(F)$ is $\mathfrak S_n$-invariant there are three
  possibilities for $[\ann(F)]_1\subseteq [S]_1 = V_0\oplus V_1$,
  where $V_0$ is the trivial representation and $V_1$ is the standard
  representation. If $[\ann(F)]_1\ne 0$, we have $[\ann(F)]_1 = V_0$
  or $[\ann(F)]_1 = V_1$. In the second case, $[\ann(F)]_1 =
  \langle p_1\rangle^\perp$ and $F = \lambda p_1^3$ and in the first
  case $[\ann(F)]_1 =[\langle p_1\rangle]_1$. There is a unique solution
  to $p_1\circ F =0$  since
  \[
    \begin{split}
    p_1\circ (a_0p_1^3+a_1np_1p_2+a_2n^2p_3) = 3na_0p_1^2 +
    a_1n^2p_2+2na_1p_1^2 + 3a_2n^2p_2 \\= (3a_0+2a_1)np_1^2 + (a_1
    +3a_2)n^2p_2.
  \end{split}
\]
  In all other points $[\ann(F)]_1= 0$.
\end{proof}

In order to find the Waring rank of all symmetric cubic forms we
need the following lemma about small orbits in $\mathbb P^{n-1}$ under
the action of $\mathfrak S_n$. We will assume that $n\ge 3$.

\begin{lemma}\label{lemma:orbits}
  If $X$ is an orbit of points in $\mathbb P^{n-1}$ under the
  action of $\mathfrak S_n$ with $|X|\le n$, we have that one of the
  following  holds
  \begin{enumerate}
    \item $|X|=1$ and $X =\{(1:1:\cdots:1)\}$.
    \item $|X|=n$ and $X =\{(\alpha:\beta:\cdots:\beta)\}$,
      $\alpha\ne\beta$.
      \item $n=3$ and $X = \{(1:\xi:\xi^2),(1:\xi^2:\xi)\}$, where
        $\xi^2+\xi+1=0$.
      \item $n=4$ and $X = \{(1:1:-1:-1), (1:-1:1:-1), (1:-1:-1:1)\}$.
  \end{enumerate}
\end{lemma}

\begin{proof}
  If $|X|>1$, we have at least one point of the form
  $(1:t_1:\dots:t_{n-1})$. If not all $t_i$ are equal, we have at
  least $n-1$ points of this form. In the case $n=3$, this gives a
  possible orbit $\{(1:t_1:t_2),(1:t_2:t_1)\}$ if $(t_2:t_1:1) =
  (1:t_1:t_2)$ or $(t_2:t_1:1) =  (1:t_2:t_1)$. In the first case,
  $t_2=1$ and we get $|X|=3$ or $|X|=1$. In the second case, we get
  $t_2^2=t_1$, $t_1^2=t_2$, so $t_1^3=t_2^3=1$. Thus we are we are
  either in case (1) or case (3). When $n=4$, we get a possible orbit
  if $X = \{(1:t_1:t_1:t_2),(1:t_2:t_1:t_2),(1:t_2:t_2:t_1)\}$. If
  $t_1\ne t_2$ we must have $(t_1:1:t_1:t_2) = (1:t_2:t_1:t_2)$ or
  $(t_1:1:t_1:t_2) = (1:t_2:t_2:t_1)$. In the first case we get
  $t_1=1$ and we are in case (2). In the second case we get
  $t_1^2=t_2^2$ and $t_1t_2=1$ which gives case (4).
  If $n\ge 5$, we have to have $(1:t_1:t_2:\cdots:t_2)$ and $t_2=1$
  which leads to the case (2).
\end{proof}

\begin{prop} \label{WR}
  The Waring rank of the symmetric cubic form $F =
  a_0p_1^3+a_1np_1p_2+a_2n^2p_3$ at $(a_0 : a_a : a_2)$  is
  \[wr(F) =
    \begin{cases}
      1 & \text{at $\mathcal P = (1:0:0)$;}\\
      n-1& \text{at $\mathcal Q = (2:-3:1)$ if $n=3$; }\\
      n&  \text{at $\mathcal C\setminus \{\mathcal P\}$  and, if $n = 3$, also at $\ell _1 \setminus \{ \mathcal Q \}$}; \\
      2(n-1)& \text{at  $\ell _2\setminus\{\mathcal P\}$}; \\
      n+1& otherwise.
    \end{cases}
  \]
\end{prop}

\begin{proof}
  The Waring rank is $1$ if and only if the Hilbert function is $(1,1,1,1)$ and we are
  at $(1:0:0)$ which is the cusp of $\mathcal C$ according to
  Lemma~\ref{lemma:curve}.

  We have that the Waring rank has to be at least $n-1$ at the flex
  point $(2:-3:1)$ and at least $n$ at all other points. If the rank
  is  $n-1$ at the flex point we would have $[\ann(F)]_2 = [I_X]_2$
  where $X$ is the support of the Waring decomposition. Since $F$ is
  symmetric, $X$ has to be a union of orbits and
  Lemma~\ref{lemma:orbits} says that we can only have $n=3$ or
  $n=4$. In the case $n=3$, we do get a Waring decomposition of $F$ as
  \[
    F =     (X_1+\xi X_2+\xi^2 X_3)^3 +     (X_1+\xi^2 X_2+\xi X_3)^3
    = 2p_1^3  - 9p_1p_2+9p_3.
  \]
  In the case $n=4$, there is no non-zero symmetric cubic that can be
  expressed in terms of the three linear forms $X_1+X_2-X_3-X_4,
  X_1-X_2+X_3-X_4,X_1-X_2-X_3+X_4$. Thus, for $n\ge 4$, the Waring
  rank has to be $n$ at the flex point.

For points on the curve $\mathcal C$  we have a power sum decomposition with $n$ terms (see Equation \eqref{eq:parametrization}). Moreover, for all these points apart from the flex point and  the cusp, the Hilbert function of $S/\ann (F)$ is $(1,n,n,1)$ by Proposition \ref{prop:HF}. Hence the Waring rank of $F$ has to be equal
  to $n$.  

  For points on the line $a_2=0$, which is the tangent line to
  $\mathcal C$ at the cusp, we have that $ F  =
  a_0p_1^3+a_1np_1p_2$. Let $Y =\{(\sum \lambda_ix_i)^2\colon
  \sum_{i=1}^n \lambda_i = \sum_{i=1}^n \lambda_i^2\}$. We have that
  $Y\subseteq\ann(F)$ and $\dim Y = n-2$ as an affine variety. Let $V
  = [\ann(F)]_2$. If $F$ has a Waring decomposition corresponding to
  the reduced set of points $X$, we have that $[I_X]_2 \cap Y
  =\{0\}$. We hence have that
\[
\dim \mathbb P(Y) + \dim \mathbb P([I_X]_2) < \dim \mathbb P([\ann(F)]_2)
\]
which gives
\[
H_X(2) > n + n-3 = 2n-3.
\]
This shows that $wr(F)\ge 2(n-1)$. We will now show that $F$ has a
Waring decomposition of length $2(n-1)$.
Let $R =
k[y_1,y_2,\dots,y_{n-1}] =   k[x_1-x_n,x_2-x_n,\dots,x_{n-1}-x_n]\subseteq S$ and let $I_Z$
  denote the ideal in $S$ generated by $(x_i-x_j)(x_i+x_j-2x_k)$, for $i\ne
  k, j\ne k$. The ideal $J$ in $R$ generated by the same forms is an
  artinian Gorenstein ideal given by the annihilator of
  $h_2(Y_1,Y_2,\dots,Y_n)$. Thus $Z$ is a zero-dimensional Gorenstein
  scheme of length $n+1$ concentrated at the point
  $(1:1:\cdots:1)$. Moreover, it is contained in $\ann(F)$. We now
  take a Waring decomposition of $h_2(Y_1,Y_2,\dots,Y_n)$ which has
  length $n-1$ since $h_2$ is a non-degenerate quadric in $n-1$
  variables. This gives a reduced subscheme of length $n-1$ and the
  cone over this is a reduced curve $\mathcal D$ consisting of $n-1$
  lines meeting at the point $(1:1:\cdots:1)$. The ideal $\ann(F)$ is
  artinian and contains $I_{\mathcal D}$. Hence Bertini's theorem says
  that we can find a quadric in $\ann(F)$ intersecting $\mathcal D$ in
  $2(n-1)$ distinct points. The ideal of these points is contained in
  $\ann(F)$ which shows that $wr(F)\le 2(n-1)$.

  It remains to show that $wr(F)=n+1$ away from $\mathcal C$ and its
  tangent line at the cusp. Through any such point, we have a line
  through the cusp, meeting the curve $\mathcal C$ in exactly one more
  point. This gives a decomposition into $n+1$ cubes of linear forms.
  Since the Hilbert function of $S/\ann(F)$ is $(1,n,n,1)$,  $wr(F) \ge
  n$ and if it is equal to $n$, the ideal of the corresponding points
  equals the ideal $\ann(F)$ in degree $2$. Hence the points are
  invariant under the action of $\mathfrak S_n$ and by
  Lemma~\ref{lemma:orbits} the only possible orbits give points on the
  curve $\mathcal C$.
\end{proof}

\begin{figure}[hbt]
  \centering
  \includegraphics[angle = -90,width=0.8\textwidth]{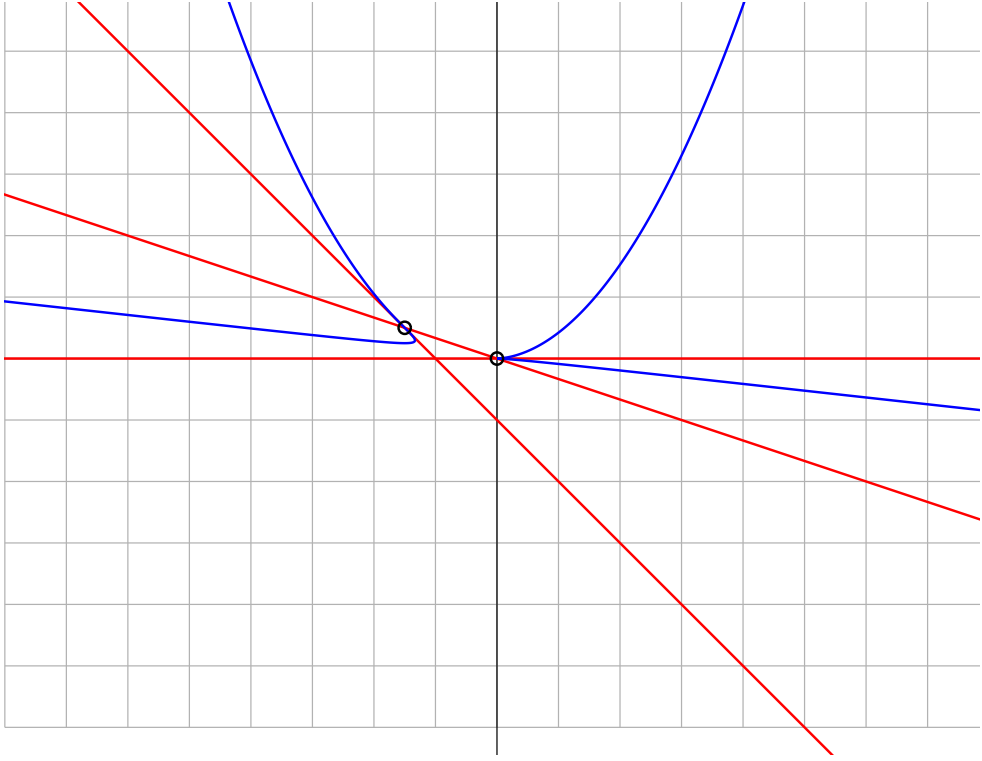}
  \caption{The cubic curve and the lines}
  \label{fig:cubic}
\end{figure}

Before looking at the cactus rank, we will study the structure of the
resolutions of $S/\ann(F)$ for symmetric cubics.

\begin{prop}\label{prop:resolutions}
  For any symmetric cubic $F = a_0 p_1^3 + a_1 n p_1p_2+ a_2 n^2 p_3$
  the Betti numbers of $S/\ann(F)$ can be obtained from the Betti numbers of
  the coordinate ring of a zero-dimensional locally Gorenstein scheme
  $X$. We can write $\ann(F) = I_X + J_X$, where $J_X$ is a canonical
  ideal and for $i =0,1,\dots,n$ and all $j$ we have that
  \[
    \beta_{i,j} (S/\ann(F))  = \beta_{i,j}(S/I_X) + \beta_{n-i,3+n-j}(S/I_X).
  \]
  There are the following cases for the resolution of $I_X$.
  \begin{itemize}
  \item[$(i)$]  For $F$ at $\mathcal P$, $X$ is a single point and the
    resolution of $S/I_X$ is linear with $\beta_{i,i} =
    \binom{n-1}{i}$, for $i=0,1,\dots,n-1$.
  \item[$(ii)$] For $F$ at $\mathcal Q$, and $n=3$, $X$ is two points
    and the resolution of $S/I_X$ is a Koszul complex with
    $\beta_{0,0}=\beta_{1,1} =\beta_{1,2}=\beta_{2,4}=1$.
  \item[$(iii)$] For $F$ at $\mathcal Q$, and $n>3$, $X$ is an 
    arithmetically Gorenstein set of $n$ points in a hyperplane and the
    resolution is forced by the $h$-vector $(1,n-2,1)$, i.e.,
    $\beta_{0,0} = \beta_{1,1} = \beta_{n-2,n} = \beta_{n-1,n+1} =
    1$ and \[
      \beta_{i,i+1} =  \frac{i(n-2-i)}{n-1}\binom{n}{i+1} +
      \frac{(i-1)(n-1-i)}{n-1}\binom{n}{i}, \quad i = 1,2,\dots,n-2.
      \]
  \item[$(iv)$] For $F$ on  $\ell _1\setminus \{ \mathcal P, \mathcal
    Q$\}, and $n=3$, $X$ is a set of three points not on a line, with a
    linear resolution, i.e., $\beta_{0,0}=1$, $\beta_{1,2}=3$ and $\beta_{2,3}=2$.
  \item[$(v)$]  For $F$ on $\ell _1\setminus \{ \mathcal P, \mathcal
    Q \}$, and $n>3$, $X$ is a set of $n+1$ points where $n$ of them are
    linearly general in a hyperplane. The resolution of $S/I_X$ is
    linear except for the last two homological degrees where $\beta_{n-2,n} =
    \beta_{n-1,n+1} = 1$. The Betti numbers are
    \[
    \beta_{i,i+1} = \frac{i(n-1-i)}{n}\binom{n+1}{i+1}, \qquad i = 1,2,\dots,n-2.
    \]
  \item[$(vi)$] For $F$ at $\mathcal C$ away from $\mathcal P$ and
    $\mathcal Q$, $X$ is a set of $n$ points not in a hyperplane. The
    resolution of $S/I_X$ is linear with $\beta_{0,0}=1$ and 
    \[
      \beta_{i,i+1} = i\binom{n}{i+1}, \qquad i=1,2,\dots,n-1.
    \]
  \item[$(vii)$]  For $F$ on $\ell _2\setminus \{ \mathcal P \}$, $X$
    is a non-degenerate  arithmetically Gorenstein scheme of
    length $n+1$ concentrated at a single point and $S/I_X$ has an
    almost linear resolution, i.e., $\beta_{0,0}=\beta_{n-1,n+1}=1$
    and
    \[
      \beta_{i,i+1} = \frac{i(n-1-i)}{n}\binom{n+1}{i+1}, \qquad  i = 1,2,\dots,n-2.
    \]
  \item[$(viii)$]  For all other $F$, $X$ is an arithmetically
    Gorenstein set of $n+1$ points and $S/I_X$ has an almost linear
    resolution, i.e., $\beta_{0,0}=\beta_{n-1,n+1}=1$
    and
    \[
      \beta_{i,i+1} = \frac{i(n-1-i)}{n}\binom{n+1}{i+1}, \qquad  i = 1,2,\dots,n-2.
    \]
  \end{itemize}
\end{prop}

\begin{proof}
  The existence of the set $X$ comes from the Waring decompositions of
  Proposition~\ref{WR} except for in case $(vii)$ and we will start by
  looking at what happens in that case. Here we have $F = a_0
  p_1^3 + a_1 n p_1p_2$ for some coefficients $a_0,a_1$, where $a_1\ne0$. Let $J$ be
  the ideal generated by the quadratic polynomials in the subring $R =
  k[x_1-x_2,x_1-x_3,\dots,x_1-x_n]$ that annihilate $p_2$. All such
  generators annihilate $p_1^3$ and $p_1p_2$, so $J\subseteq \ann(F)$
  for any $F = a_0 p_1^3 + a_1 n p_1p_2$. Furthermore, $A = R/(R\cap J)$
  is artinian with Hilbert function $(1,n-1,1)$. It is Gorenstein,
  since there is no socle in degree one. Indeed, $(\sum \alpha_i
  x_i)(\sum \beta_i x_i)\circ p_2 = 2\sum \alpha_i\beta_i$ which gives
  a non-degenerate pairing on $R_1$. Thus we have that $J = I_X$ where
  $X$ is a zero-dimensional arithmetically Gorenstein subscheme
  concentrated at the point $(1:1:\cdots:1)$. In fact, $A$ is the
  subring of $S/\ann(p_2)$ generated by
  $x_1-x_2,x_1-x_3,\dots,x_1-x_n$. The symmetric quadric $q$ in
  $\ann(F)$ is given by
  \[
    q = (3a_0+(n+2)a_1)p_1^2 - 3n(a_0+a_1)p_2
  \]
  and evaluation at $(1,1,\dots,1)$ gives
  \[
    q(1,1,\dots,1)  = (3a_0+(n+2)a_1)n^2 - 3n(a_0+a_1)n = a_1 n^2(n-1)
    \ne 0
  \]
  showing that $q$ is a non-zero-divisor on $X$. Since $X$ is
  arithmetically Gorenstein, the ideal $(q)$ is a canonical ideal and
  we have $\ann(F) = I_X + (q)$.

  In all the other cases, the Castelnuovo--Mumford regularity of the set of points $X$ is at most $2$
  and we can apply \cite[Theorem 3.4]{B} to conclude that the image of
  the ideal $\ann(F)$ in the coordinate ring of $X$ is a canonical
  ideal. When the Castelnuovo--Mumford regularity of $X$ is at most one, the statement about the
  Betti numbers follows from \cite[Proposition 3.5]{B}. In the cases where the
  Castelnuovo--Mumford regularity is two and the scheme $X$ is arithmetically Gorenstein,
  the canonical ideal is principal, which means that it is generated
  by a non-zero-divisor. Hence the statement about the Betti numbers
  holds also in this case.

  In case $(vi)$, $S/\ann(F)$ is the quotient of the coordinate ring
  of $X$  by a non-zero-divisor and the ideal is still a canonical
  ideal since $X$ is arithmetically Gorenstein.

  It remains to treat the case $(v)$ where the Castelnuovo--Mumford regularity of $X$ is
  two but the ideal is not given by a non-zero-divisor. We look at the
  short exact sequence
  \[
    0 \longrightarrow J_X \longrightarrow S/I_X \longrightarrow
    S/(I_X+J_X) \longrightarrow 0
  \]
  where $J_X$ is the canonical ideal that is the image of $\ann(F)$ in
  $S/I_X$. Since there may be cancellations, we only get inequalities
  \[
    \beta_{i,j} (S/\ann(F)) \le  \beta_{i,j}(S/I_X) + \beta_{n-i,3+n-j}(S/I_X)
  \]
  but in our case, the only possible cancellation would be in degree
  $3$ and dually in degree $n$. If this cancellation occurred, it would
  mean that the ideal $\ann(F)$ was generated by quadrics, but this is
  not possible, since the symmetric quadric $q$ is not a
  non-zero-divisor on $X$ since it vanishes at the point
  $(1:1:\cdots:1)$.

  For each of the cases, we will now prove the statements about the
  resolutions of $S/I_X$.

  \begin{itemize}
  \item[$(i)$] For $F  = p_1^3$, $X$ is the single point
    $(1:1:\cdots:1)$ and the ideal is generated by the linear forms $x_1-x_2,
   x_1-x_3,\dots,x_1-x_n$. Hence the resolution is linear and
    given by a Koszul complex of length $n-1$, which gives the stated
    Betti numbers.
  \item[$(ii)$] Here the Waring rank is two and we have that the
    ideal of two points in $\mathbb P^2$ is a
    complete intersection of type $(1,2)$. The resolution is a Koszul
    complex which gives the stated Betti numbers.
  \item[$(iii)$] The Waring rank is $n$ and the set of points giving the
    Waring decomposition is the $\mathfrak S_n$-orbit of
    $(1-n:1:1:\cdots:1)$. They lie in a hyperplane and are in
    linearly general position in that hyperplane. Hence they form an
    arithmetically Gorenstein set of points with $h$-vector
    $(1,n-2,1)$. The Betti numbers of the coordinate ring $S/I_X$ are
    given as sums of the consecutive Betti numbers of an artinian
    Gorenstein algebra with the same $h$-vector. This means that
    $\beta_{0,0} = \beta_{1,1} = \beta_{n-2,n} = \beta_{n-1,n+1} =
    1$ and \[
      \beta_{i,i+1} =  \frac{i(n-2-i)}{n-1}\binom{n}{i+1} +
      \frac{(i-1)(n-1-i)}{n-1}\binom{n}{i}, \quad i = 1,2,\dots,n-2,
    \]
    where we get the Betti numbers for the artinian Gorenstein algebra
    with $h$-vector $(1,n-2,1)$ by the formula for the Betti numbers
    of a pure resolution (\cite{HK}).
  \item[$(iv)$] We have that $X = \{ (1:1:1), (1:\xi:\xi^2),
    (1:\xi^2,\xi)\}$, where $\xi^2+\xi+1=0$ and the three points do
    not lie on a line. For three points not on a line in $\mathbb
    P^2$, the resolution is linear and the Betti numbers are
    $\beta_{0,0} = 1$,  $\beta_{1,2} = 3$ and $\beta_{2,3}= 2$.
  \item[$(v)$] $X$ is the union of the single point orbit
    $X_1 = \{(1:1:\cdots:1)\}$ and the $n$-point orbit $X_2$ of
    $(1-n:1:1:\cdots:1)$. Since $X_2$ lies in a hyperplane, but is in
    linearly general position within that hyperplane, $X_2$ is
    arithmetically Gorenstein with $h$-vector $(1,n-2,1)$. Thus the
    $h$-vector of $X$ has to be $(1,n-1,1)$. For an artinian reduction
    $B = S/J$ of $S/I_X$, the ideal in degree $2$ equals the ideal of an
    artinian reduction of $S/I_{X_2}$. Thus we see that $B$ has a
    one-dimensional socle in degree $2$ and a one-dimensional socle in
    degree $1$. This gives us the last column in the Betti table of
    $B$ and forces that $\beta_{n-2,2}=1$. By duality, we get that
    $\beta_{i,i+2}=0$ for $i=0,1,\dots,n-3$. Hence all the remaining
    Betti numbers are equal to the Betti numbers of an arithmetically Gorenstein
    scheme with the same  $h$-vector as $X$, i.e., the same as in case
    $(vii)$ below.
  \item[$(vi)$] Here $X$ is the $\mathfrak S_n$-orbit of
    $(n\alpha+\beta:\beta:\cdots:\beta)$ where $\alpha\ne 0$ and
    $\alpha+\beta\ne0$. Thus $X$ is a set of $n$ points not in a
    hyperplane with $h$-vector $(1,n-1)$ showing that the resolution
    is linear. Again, we can use the formula for the Betti numbers of
    a pure resolution to write
    \[
      \beta_{i,i+1} = i\binom{n}{i+1}, \qquad i=1,2,\dots,n-1.
    \]
  \item[$(vii)$] As we have seen, $X$ is an arithmetically Gorenstein
    subscheme of length $n+1$ concentrated at the point
    $(1:1:\cdots:1)$. The $h$-vector is $(1,n-1,1)$ and the resolution
    of $S/I_X$  is almost linear since the artinian reduction is
    extremely compressed (\cite{FL}). From the formula for the
    Betti numbers of a pure resolution, we get
    \[
      \beta_{i,i+1} = \frac{i(n-1-i)}{n}\binom{n+1}{i+1}, \qquad i=1,2,\dots,n-2.
    \]
  \item[$(viii)$] Here $X$ is the union of the single point orbit $X_1
    = \{(1:1:\cdots:1)\}$ and the $n$-point orbit of
    $\{(n\alpha+\beta:\beta:\cdots:\beta)$ where $\alpha\ne 0$ and
    $\alpha+\beta\ne 0$. This set of points is in linearly general
    position with $h$-vector $(1,n-1,1)$. Thus $X$ is arithmetically
    Gorenstein and $S/I_X$ has an almost linear resolution as in the
    case $(vii)$ above.
  \end{itemize}
\end{proof}

We now turn to the cactus rank and we will also use the following
result by K. Ranestad and F. -O. Schreyer.

\begin{prop}[{\cite[Corollary 1]{RS}}]
    \label{prop:lower bound cactus}
Let $F$ be a homogeneous polynomial. If the degree of any minimal generator of $\ann (F)$ is at most $d$, then one has
\[
cr (F) \ge \frac{1}{d} \cdot \deg( \ann(F)).
\]
\end{prop}

\begin{prop}
  \label{prop:cr}
  The cactus rank of the symmetric cubic form $F =
  a_0p_1^3+a_1np_1p_2+a_2n^2p_3$ at $(a_0 : a_a : a_2)$ is
  \[
  cr(F) =
    \begin{cases}
      1 & \text{at $\mathcal P = (1:0:0)$}; \\
      n-1 & \text{at $\mathcal Q = (2:-3:1)$, provided $n = 3$};\\
       n& \text{at $\mathcal C\setminus \{\mathcal P, \mathcal Q\}$  and, if $n = 3$, also  along  $\ell _1 \setminus \{ \mathcal Q \}$}; \\
       n+1& otherwise.
    \end{cases}
  \]
\end{prop}

\begin{proof}
We will combine the lower bound provided by Proposition \ref{prop:lower bound cactus} and the upper bound $cr (F) \le wr (F)$, using that we know the Waring rank by Proposition \ref{WR}.
We consider cases depending on the Hilbert function of $S/\ann (F)$.

By Proposition \ref{WR}, we know $wr (F) = 1$ if and only if $F$ is the form at $\mathcal P$, i.e., $F = p_1^3$, which has cactus rank one.

Now consider $F$ at the flex point $\mathcal Q$ of $\mathcal C$. If $n
= 3$ then $\ann (F)$ is a complete intersection of type $(1, 2, 3)$ by
Proposition \ref{prop:resolutions}(ii). Hence, this ideal does not
contain the ideal of a point, which implies $cr (F) \ge 2$. Since $wr
(F) = 2$ by Proposition \ref{WR}, we get $cr (F) =2$. If $n \ge 4$,
then $\ann (F)$ is generated by quadrics (see Proposition
\ref{prop:resolutions}(iii)) and has degree $2n+2$  by Proposition
\ref{prop:HF}. Hence Proposition \ref{prop:lower bound cactus} gives
$cr (F) \ge \frac{2 n +2}{2} = n+1$. We get equality because
Proposition \ref{WR} yields $wr (F) = n+1$.

It remains to consider cubics at points other than the cusp $\mathcal P$ and  the flex point
$\mathcal Q$. Hence, the Hilbert function of $S/\ann (F)$ is $(1, n, n, 1)$  by Proposition \ref{prop:HF}. In particular, $\ann(F)$ has degree $2n+2$. Observe that $\ann (F)$ is generated by quadrics, unless $F$ corresponds to a point on the curve $\mathcal C$ or a point on the line
$\ell _1$. Excluding the latter two cases, Proposition \ref{prop:lower bound cactus} gives $cr (F) \ge \frac{2 n +2}{2} = n+1$. We get equality if $F$ is also not on the line
$\ell _2$ because then $wr (F) = n+1$ by Proposition \ref{WR}. Thus, we are left with considering three situations.

First, assume that  $F$ is a point on   $(\ell_1\cup \mathcal C) \setminus \{\mathcal P, \mathcal Q\}$. If $Z \subset \mathbb P^{n-1}$ is a zero-dimensional subscheme such that its ideal $I_Z$ is contained in $\ann (F)$, then we get $\dim_K [S/I_Z]_2 \ge \dim_K [S/\ann (F)]_2 = n$, where the equality is due to Proposition \ref{prop:HF}. We conclude that $\deg Z \ge \dim_K [S/I_Z]_2 \ge n$, which implies $cr (F) \ge n$. Unless $F$ is a point on the line  $\ell _1$ and $n \ge 4$, we have
 $wr (F) = n$ by Proposition \ref{WR}. Thus we get $cr (F) = n$, as desired.

Now consider the case $n \geq 4$ and $F$ is on the line $\ell _1$. Suppose that $cr(F) = n$.
This means there is a zero-dimensional subscheme $Z \subset \mathbb P^{n-1}$  of degree $n$ such that  $I_Z \subset \ann (F)$. Let $J$ be the ideal generated by the quadrics in $\ann (F)$. Comparing Hilbert functions, we see that $[I_Z]_2 = [J]_2 = [\ann (F)]_2$.
By Proposition \ref{prop:resolutions}(v), we get that $\beta_{1,3}(S/\ann(F)) =1$
and hence $\ann (F)$ is generated by
$J$ together with one cubic. Thus by Proposition \ref{prop:HF} the value of the Hilbert function of $R/J$ in degree 3 is $n+1$.
Since $Z$ has degree $n$, its Hilbert function must be $1, n, n, \ldots$. It follows that the Castelnuovo-Mumford regularity of $I_Z$ is two. Hence $I_Z$ is generated by quadrics, as is $J$. This implies $I_Z = J$, which is a contradiction because the Hilbert function of $R/I_Z$ in degree 3 is $n$.
This proves $cr (F) \ge n+1$. We get equality because $wr (F) = n+1$ by Proposition \ref{WR}.

Second, assume $F$ is a point on the line  $\ell _2$, but not  $\mathcal P$.
Since $\ell _2$ meets $\mathcal C$ only in $\mathcal P$, the ideal $\ann (F)$ is
generated by quadrics (see Proposition \ref{prop:resolutions}). Hence Proposition \ref{prop:lower bound cactus} gives $cr (F) \ge n+1$, as above. In the  proof of
Proposition \ref{prop:resolutions}(vii), we constructed a Gorenstein scheme $X$ of degree
$n+1$ supported at one point with the property that $I_X \subset \ann
(F)$.
Hence, by
definition of the cactus rank, we get $cr (F) \le \deg X = n+1$. Together with the lower bound this yields $cr (F) = n+1$.
\end{proof}

For the cactus rank of general cubics, the reader may look at
\cite{BR}. Now we show that the cactus rank and the Waring rank of
symmetric cubics are usually equal, but that the difference in some
cases can be made arbitrary large when increasing the number of variables. 

\begin{cor} \label{compare}
The cactus rank and the Waring rank of any symmetric cubic $F$ in $n \ge 3$ variables are equal unless $F =
  a_0p_1^3+a_1np_1p_2$ with $a_1\ne 0$ and $n \ge 4$.  In this case we have $cr (F) = n+1 < 2 (n-1) = wr (F)$.
\end{cor}

\begin{proof}
Compare Propositions \ref{WR}  and \ref{prop:cr}.
\end{proof}

\begin{prop} \label{cubic SLP}
For any symmetric cubic form $F =
  a_0p_1^3+a_1np_1p_2+a_2n^2p_3$, $S/\operatorname{ann}(F)$ satisfies the SLP.
\end{prop}

\begin{proof} If we are at the cusp $\mathcal P = (1:0:0)$ of $\mathcal C$, $F=p_1^3$, $wr(F)=1$, the Hilbert function of $S/\operatorname{ann}(F)$ is $(1,1,1,1)$, $[\operatorname{ann}(F)]_1=\langle x_i-x_j, i\ne j\rangle _1$ and clearly $S/\operatorname{ann}(F)$ has the SLP.

If we are at the flex $\mathcal Q = (2:-3:1)$ of $\mathcal C$ i.e. $F=2p_1^3-3np_1p_2+n^2p_3$ and $n=3$ we have $wr(F)=2$, the Hilbert function of $S/\operatorname{ann}(F)$ is $(1,2,2,1)$ (Proposition \ref{prop:HF}) and  $\operatorname{ann}(F)$ is a complete intersection artinian ideal of codimension 2. Hence, it has the SLP (see \cite[Proposition 4.4]{HMNW}).
 If we are at the flex $\mathcal Q$ of $\mathcal C$ and $n>3$,  we have $wr(F)=n$ (Proposition \ref{WR}), the Hilbert function of $S/\operatorname{ann}(F)$ is $(1,n-1,n-1,1)$ (Proposition \ref{prop:HF}), we easily check that $\{ x_i-x_1\}_{i=2,\cdots ,n}$ is a basis of $S/\operatorname{ann}(F)$ and $x_1$ is an SL element since $x_1^{3-2i}:[S/\operatorname{ann}(F)]_i\longrightarrow  [S/\operatorname{ann}(F)]_{3-i}$  is an isomorphism for $i=0,1$.

Assume  that $(a_0,a_1,a_2)$ is neither the flex nor the cusp of $\mathcal C$.
 According to Proposition \ref{prop:HF}
  the Hilbert function of $S/\operatorname{ann}(F)$ is
      $(1,n,n,1)$ and we distinguish two cases:

\vskip 2mm
 \noindent Case 1: $wr(F)=n$. In this case, $[\operatorname{ann}(F)]_2=[I_X]_2$ where $X$ is the support of the Waring decomposition of $F$. Therefore, any non-zero divisor on $X$ is an SL element.

\vskip 2mm
 \noindent Case 2: $wr(F)>n$.  We will show that 3 linear forms are
 sufficient to provide an SL elements for all of them. To prove that $\ell $ is an SL element for $S/\operatorname{ann}(F)$ it suffices to show that  $\ell :[S/\operatorname{ann}(F)]_1\longrightarrow  [S/\operatorname{ann}(F)]_{2}$  is injective. The map $\ell :[S/\operatorname{ann}(F)]_1\longrightarrow  [S/\operatorname{ann}(F)]_{2}$  fails injectivity if there is a linear form $g\in S$ such that $\ell g=0$ in $[S/\operatorname{ann}(F)]_{2}$, i.e. there is a linear form $g\in S$ such that $\ell g \circ F=0$. So, our strategy is to study when the quadratic form $\ell \circ F$ drops rank.
 We first observe that $(a_0,a_1,a_2)\notin \mathcal C$.
 We   consider $\ell =\sum _{i=1}^nx_i$ and we will check when $\ell $ is an SL element for $S/\operatorname{ann}(F)$. We claim that quadratic form
    \[
    \begin{array}{rcl}
    \ell \circ (a_0p_1^3+a_1np_1p_2+a_2n^2p_3) & = & 3na_0p_1^2 +
    a_1n^2p_2+2na_1p_1^2 + 3a_2n^2p_2 \\
    & = &  (3a_0+2a_1)np_1^2 + (a_1+3a_2)n^2p_2
  \end{array}
\]
   drops rank  exactly when $(a_0,a_1,a_2)\in \ell_1\cup \ell _3$. Indeed, the matrix 
   \[M_q=
     \left[\begin{smallmatrix} (a_1+3a_2)n^2+(3a_0+2a_1)n & (3a_0+2a_1)n & (3a_0+2a_1)n & \cdots & (3a_0+2a_1)n \\
  (3a_0+2a_1)n & (a_1+3a_2)n^2+(3a_0+2a_1)n &  (3a_0+2a_1)n & \cdots & (3a_0+2a_1)n \\
  (3a_0+2a_1)n & (3a_0+2a_1)n & (a_1+3a_2)n^2+(3a_0+2a_1)n &   \cdots & (3a_0+2a_1)n \\
  \vdots & \vdots & \vdots & \vdots & \vdots \\
   (3a_0+2a_1)n & (3a_0+2a_1)n &    (3a_0+2a_1)n & \cdots & (a_1+3a_2)n^2+(3a_0+2a_1)n
   \end{smallmatrix}\right]
   \]
   associated to the quadratic form $q=  (3a_0+2a_1)np_1^2 + (a_1+3a_2)n^2p_2$ has determinant
   \[\det (M_q)=3n^{2n}(a_1+3a_2)^{n-1}(a_0+a_1+a_2).\]
   Therefore, the quadratic form $ \ell \circ
   (a_0p_1^3+a_1np_1p_2+a_2n^2p_3)$ drops rank exactly when
   $(a_0,a_1,a_2)\in \ell_1\cup \ell _3$,  and we conclude that $\ell
   =\sum _{i=1}^nx_i$ is an SL element except on the lines $\ell _1$ and $\ell _3$.
   Finally, we consider the quadratic forms
   \[
     q_1:=x_1 \circ (a_0p_1^3+a_1np_1p_2+a_2n^2p_3)  =  3a_0p_1^2 +
     2a_1nx_1p_1+3a_2n^2x_1^2 + a_1np_2
   \]
   and
   \[
     q_2:= (nx_1-p_1) \circ (a_0p_1^3+a_1np_1p_2+a_2n^2p_3)  = -2a_1np_1^2 +
     2a_1n^2x_1p_1+3a_2n^3x_1^2 -3a_2n^2p_2.
   \]
   Computing the determinant of the symmetric matrices associated to $q_1$ and $q_2$ we get that the quadratic form $q_1$ and $q_2$  never drop rank simultaneously  in points $(a_0,a_1,a_2)$ of the lines $\ell _3$ and $\ell _1$.
   Therefore on these lines either   $\ell=x_1$ or $\ell=nx_1-\sum
   _{i}x_ i$ is an SL element.
 \end{proof}
 
 \section{Generic Waring rank for symmetric forms}\label{section:genWR}
 As we have seen in the previous section, the generic symmetric cubic
 form in $n$ variables has Waring rank $n+1$. We will now see that we
 can get bounds for the Waring rank of generic symmetric quartics and
 quintics using the same kind of orbits that come into the
 decompositions of Theorem~\ref{thm:powersums}. We start by looking at
 the power sum expansion
 \[
   F(\alpha_0,\alpha_1,\alpha_2)  = \alpha_0^3 h_1^3 + \sum_{i=1}^n
   (\alpha_1h_1+\alpha_2 X_i)^3 =\alpha_2^3 p_3 + 3\alpha_1\alpha_2^2
   p_1p_2 + (\alpha_0^3+n\alpha_1^3+3\alpha_1^2\alpha_2)p_1^3
 \]
 which gives rise to a rational map from $\mathbb P^2$ with
 coordinates $(\alpha_0\colon\alpha_1\colon\alpha_2)$ to $\mathbb
 P_2(\langle p_3,p_2p_1,p_1^3\rangle)$. In order to see that this map is
 generically onto, we look at the Jacobian
 \[
   \frac{\partial F_i}{\partial \alpha_j} = \begin{bmatrix}
     0 & 0 & 3 \alpha_0^2\\
     0 & 3\alpha_2^2 & 3n\alpha_1^2+6\alpha_1\alpha_2\\
     3\alpha_2^2 & 6 \alpha_1\alpha_2 & 3\alpha_1^2\\
   \end{bmatrix}
 \]
 with determinant equal to $27\alpha_0^2\alpha_2^4$. This is non-zero
 except when $\alpha_0=0$, which corresponds to the cubic curve $\mathcal C$
 and when $\alpha_2=0$, which corresponds to the point $\mathcal P$,
 i.e., the cusp of $\mathcal C$.
 
 \begin{prop}\label{prop:generic_rank}
   The generic symmetric quartic has Waring rank at most $\binom{n+1}{2}+1 =
   \binom{n+2}{2}-n$ and the generic symmetric quintic has Waring rank at
   most $\binom{n+2}{2}$. 
 \end{prop}
 
 \begin{proof}
   We start by using the decomposition
   \[
     F(\alpha_0,\alpha_1,\alpha_2,\alpha_3,\alpha_4)  = \alpha_0^4 h_1^4 + \sum_{i=1}^{n}
     (\alpha_1 h_1+\alpha_2 X_i)^4
     + \sum_{1\le i < j \le n} (\alpha_3 h_1 + \alpha_4 X_i+\alpha_4
     X_j)^4,
   \]
   which produces a symmetric quartic form of Waring rank at most $1 +
   n + \binom{n}{2} = \binom{n+1}{2}+1 = \binom{n+2}{2}-n$. We expand this
   into the basis of symmetric quartics given by
   $\{p_4,p_3p_1,p_2^2,p_2p_1^2,p_1^4\}$ to get
   \[
     \begin{split}
       F(\alpha_0,\alpha_1,\alpha_2,\alpha_3,\alpha_4)  = (\alpha_2^4 +
       (n-8)\alpha_4^4)p_4 + (4\alpha_1\alpha_2^3 + 4
       (n-4)\alpha_3\alpha_4^3 + 4 \alpha_4^4)p_3p_1 +
       3 \alpha_4^4 p_2^2 \\+ (6\alpha_1^2\alpha_2^2 +
       6(n-2)\alpha_3^2\alpha_4^2 + 12 \alpha_3\alpha_4^3)p_2p_1^2 \\+
       (\alpha_0^4 + n \alpha_1^4 + 4 \alpha_1^3\alpha_2 + \binom{n}{2}
       \alpha_3^4 + 4(n-1)\alpha_3^3\alpha_4  + 6 \alpha_3^2\alpha_4^2)p_1^4.
     \end{split}
   \]
   Thus we have produced a rational map from $\mathbb P^4$ with
   coordinates
   $(\alpha_0\colon\alpha_1\colon\alpha_2\colon\alpha_3\colon\alpha_4)$
   to $\mathbb P^4 = \mathbb P(\langle
   p_4,p_3p_1,p_2^2,p_2p_1^2,p_1^4\rangle)$. In order to show that this
   is generically onto, which proves our statement, we consider the
   Jacobian
   \[
     \left(\frac{  \partial F_j}{\partial \alpha_i}\right)  =
     \left[\begin{smallmatrix}
         0& 0& 0& 0& 4\alpha_0^3\\
         0& 4\alpha_2^3& 0& 12\alpha_1\alpha_2^2&4n\alpha_1^3+12\alpha_1^2\alpha_2\\
         4\alpha_2^3& 12\alpha_1\alpha_2^2& 0& 12\alpha_1^2\alpha_2& 4\alpha_1^3\\
         0&4(n-4)\alpha_4^3& 0& 12(n-2)\alpha_3\alpha_4^2+12\alpha_4^3&2n(n-1)\alpha_3^3+12(n-1)\alpha_3^2\alpha_4+12\alpha_3\alpha_4^2\\
         4(n-8)\alpha_4^3&12(n-4)\alpha_3\alpha_4^2+16\alpha_4^3& 12\alpha_4^3&12 (n-2)\alpha_3^2\alpha_4+36\alpha_3\alpha_4^2&
         4(n-1)\alpha_3^3+12\alpha_3^2\alpha_4
       \end{smallmatrix}\right]
   \]
   which has determinant
   \[\begin{split}
       \det  \left(\frac{  \partial F_j}{\partial \alpha_i}\right)  =
       (4\alpha_0^3)(4\alpha_2^3)(12\alpha_4^3)\det \left[\begin{matrix}4\alpha_2^3&
           12\alpha_1\alpha_2^2\\4(n-4)\alpha_4^3 &
           12(n-2)\alpha_3\alpha_4^2+12\alpha_4^3\end{matrix} \right]\\
       =
       2^{10}3^2\alpha_0^3\alpha_2^5\alpha_4^5((n-2)\alpha_2\alpha_3-(n-4)\alpha_1\alpha_4+\alpha_2\alpha_4).
     \end{split}
   \]
   This is generically non-zero, showing that the generic symmetric quartic
   has power sum expansion with at most $\binom{n+1}{2}+1$ terms.
   
   For the symmetric quintics, we use the power sum expansion
   \[\begin{split}
       F(\alpha_0,\alpha_1,\alpha_2,\alpha_3,\alpha_4,\alpha_5,\alpha_6)  = \alpha_0^5 h_1^5 + \sum_{i=1}^{n}
       (\alpha_1 h_1+\alpha_2 X_i)^5
       + \sum_{1\le i < j \le n} (\alpha_3 h_1 + \alpha_4 X_i+\alpha_4
       X_j)^5 \\ +\sum_{i=1}^{n}
       (\alpha_5 h_1+\alpha_6 X_i)^5
     \end{split}
   \]
   which gives a rational map from $\mathbb P^6$ to $\mathbb P^6 =
   \mathbb P(\langle
   p_5,p_4p_1,p_3p_2,p_3p_1^2,p_2^2p_1,p_2p_1^3,p_1^5\rangle)$ given by
   the expansion
   \[\begin{split}
       F(\alpha_0,\alpha_1,\dots,\alpha_6)  =
       (\alpha_2^5+(16-n)\alpha_4^5+\alpha_6^5)p_5 +
       (5\alpha_1\alpha_2^4+5(n-8)\alpha_3\alpha_4^4 +
       5\alpha_4^5+5\alpha_5\alpha_5^4)p_4p_1 \\+ 10\alpha_4^5p_3p_2 +
       (10\alpha_1^3\alpha_2^3+10(n-4)\alpha_3^2\alpha_4^3 + 20
       \alpha_3\alpha_4^4+10\alpha_5^2\alpha_6^3)p_3p_1^2 \\+
       15\alpha_3\alpha_4^4p_2^2p_1+(10\alpha_1^3\alpha_2^2+10(n-2)\alpha_3^3\alpha_4^2+30\alpha_3^2\alpha_4^3+10\alpha_5^3\alpha_6^2)p_2p_1^3\\+
       (\alpha_0^5+n\alpha_1^5 + 5\alpha_1^4\alpha_2+5(n-1)\alpha_3^4\alpha_4
       +10\alpha_3^3\alpha_4^2+5\alpha_5^4\alpha_6+\binom{n}{2}\alpha_3^5)p_1^5
       .
     \end{split}
   \]
   
   The Jacobian of this map is given by
   \[\left[
       \begin{smallmatrix}
         0&0&0&0&0&0&5\alpha_0^4\\
         0&5\alpha_2^4&0&20\alpha_1\alpha_2^3&0&30\alpha_1^2\alpha_2^2&5n\alpha_1^4+20\alpha_1^3\alpha_2\\
         5\alpha_2^4&20\alpha_1\alpha_2^3&0&30\alpha_1^2\alpha_2^2&0&20\alpha_1^3\alpha_2&5\alpha_1^4\\
         0&5(n-8)\alpha_4^4&0&20(n-4)\alpha_3\alpha_4^3+20\alpha_4^4&15\alpha_4^4&30(n-2)\alpha_3^2\alpha_4^2+60\alpha_3\alpha_4^3&5\binom{n}{2}\alpha_3^4+20(n-1)\alpha_3^3\alpha_4+30\alpha_3^2\alpha_4^2\\
         5(n-16)\alpha_4^4&20(n-8)\alpha_3\alpha_4^3+25\alpha_4^4&50\alpha_4^4&30(n-4)\alpha_3^2\alpha_4^2+80\alpha_3\alpha_4^3 &60\alpha_3\alpha_4^3 &20(n-2)\alpha_3^3\alpha_4+90\alpha_3^2\alpha_4^2&5(n-1)\alpha_3^4\alpha_4+20\alpha_3^3\alpha_4\\
         0&5\alpha_6^4&0&20\alpha_5\alpha_6^3&0&30\alpha_5^2\alpha_6^2&5n\alpha_5^4+20\alpha_5^3\alpha_6\\
         5\alpha_6^4&20\alpha_5\alpha_6^3&0&30\alpha_5^2\alpha_6^2&0&20\alpha_5^3\alpha_6&5\alpha_5^4\\
       \end{smallmatrix}
     \right]
   \]
   and we get the determinant of that as
   \[
     (5\alpha_0^4)(15\alpha_4^4)(50\alpha_4)^4\det\left[
       \begin{smallmatrix}
         0&5\alpha_2^4&20\alpha_1\alpha_2^3&30\alpha_1^2\alpha_2^2\\
         5\alpha_2^4&20\alpha_1\alpha_2^3&30\alpha_1^2\alpha_2^2&20\alpha_1^3\alpha_2\\
         0&5\alpha_6^4&20\alpha_5\alpha_6^3&30\alpha_5^2\alpha_6^2\\
         5\alpha_6^4&20\alpha_5\alpha_6^3&30\alpha_5^2\alpha_6^2&20\alpha_5^3\alpha_6\\
       \end{smallmatrix}\right]
     = 2^3\cdot 3\cdot 5^9\alpha_0^4\alpha_2^4\alpha_4^8\alpha_6^4(\alpha_1\alpha_6-\alpha_2\alpha_5)^4.
   \]
   This is generically non-zero. In fact, it is non-zero as long as the
   linear forms involved in the decomposition are distinct. Thus
   the generic symmetric quintic form has Waring rank at most
   $\binom{n+2}{2}$.    
 \end{proof}
 
 \begin{rmk}
   In the case of cubics, Proposition~\ref{WR} shows that the generic rank
   is indeed $n+1$. In the case of quartics, we know from the Hilbert
   function that the generic rank is at least $\binom{n+1}{2}$ and our
   upper bound is just one more. However, we are not able to use this
   approach to prove that the generic rank of symmetric quartics
   cannot be $\binom{n+1}{2}$. Observe that the generic symmetric
   form might have a Waring decomposition that is not invariant under
   the  action of the symmetric group.  In the case of quintics, the
   lower bound given by the Hilbert function is again
   $\binom{n+1}{2}$ and our upper bound is $n+1$ higher than
   this. Here we might be able to use the resolution in order to show
   that the generic rank is higher than $\binom{n+1}{2}$, but we have
   not been able to do this in general.

   For symmetric forms of degree six and higher, we cannot expect that
   the generic symmetric form can be expanded into powers of linear
   forms in a similar way as the complete symmetric form. This can be
   seen from looking at the dimension of the family of symmetric forms
   compared to the number of parameters that can be involved in the
   expansions. For example, in the case of sextics, the dimension of
   the family is ten while an expansion corresponding to the
   expansion we have for quartics would give only a nine-dimensional
   family. For higher degrees, the difference in dimensions grows
   larger and larger.
 \end{rmk}

 \begin{rmk}
   We can now look more closely at the claims of
   Remark~\ref{rmk:bounds}. Using the rational map from $\mathbb P^4$
   to $\mathbb P^4$ described in the proof of Proposition~\ref{prop:generic_rank} we can look at the preimage of $h_{n,4}$. Using the fact that
\[
h_{n,4} = \frac{1}{24}(6 p_4 + 8 p_3p_1 + 3p_2^2 + 6 p_2p_1^2 + p_1^4)
\]
we get that the preimage is given by the following  system of equations
\[
  \left\{
  \begin{array}{rcl}
    \alpha_2^4+(n-14)\alpha_4^4&=&0 \\
    \alpha_1\alpha_2^3-\alpha_4^4+(n-4)\alpha_3\alpha_4^3&=& 0\\
    \alpha_1^2\alpha_2^2+2\alpha_3\alpha_4^3+(n-2)\alpha_3^2\alpha_4^2-\alpha_4^4&=& 0\\
    \alpha_0^4+4\alpha_1^3\alpha_2+n\alpha_1^4+6\alpha_3^2\alpha_4^2+4(n-1)\alpha_3^3\alpha_4+n(n-1)/2\alpha_3^4-\alpha_4^4&=& 0. \\
  \end{array}\right.
\]

From the first three equations, we get
\[
  \alpha_1^2\alpha_2^6 = (\alpha_4^4-(n-4)\alpha_3 \alpha_4^3)^2 = (14-n)\alpha_4^4(\alpha_4^4-2\alpha_3\alpha_4^3-(n-2)\alpha_3^2\alpha_4^2) 
\]
which gives 
\[
  \alpha_4^6(\alpha_4^2n-4\alpha_4\alpha_3n+8\alpha_3^2n-13\alpha_4^2+36\alpha_4\alpha_3-12\alpha_3^2)
  = 0.
\]
For all $n$ we can solve this equation with $\alpha_4\ne 0$ and hence
get solutions for $\alpha_0$, $\alpha_1$ and  $\alpha_2$ in terms of
$\alpha_4$. Because of the first equation, we see that $\alpha_2\ne
0$ when $n\ne 14$. When $n=14$, we get that $\alpha_2=0$ from the first
equation and then we can see that the Waring rank of
$F(\alpha_0,\alpha_1,0,\alpha_3,\alpha_4)$ is less than the Waring
rank of $h_{14,4}$ showing that the solution to our equations give
$F(\alpha_0,\alpha_1,0,\alpha_3,\alpha_4)=0$.

Generic rank in the case $n=14$ is $170$ which is far larger than $120$
given by the decomposition in our theorem. However, it is unclear if
we can find Waring decompositions with fewer terms. If we look for
symmetric Waring decompositions, we need to take a union of orbits and
it will not be sufficient to take $n$-point orbits since such
decompositions do not give any contribution to the coefficient of $p_2^2$. 

For $n=6$, the generic rank is lower than the bound given by our
construction, $21$ instead of $22$. Also for $n<6$, generic rank is lower
by one. For $n>6$, generic rank is always higher.
 \end{rmk}
 

\end{document}